\documentclass{article}
\usepackage{hyperref}
\usepackage{cite}
\usepackage{indentfirst}
\usepackage{amsthm}
\usepackage{amsfonts}
\usepackage{amssymb}
\usepackage{xcolor}
\usepackage{tensor}
\usepackage{graphicx}
\usepackage[export]{adjustbox}
\usepackage{amsthm}
\usepackage{overpic}
\usepackage{enumerate}

\newtheorem{lemma}{Lemma}
\newtheorem{theorem}{Theorem}
\newtheorem{claim}{Claim}
\newtheorem{remark}{Remark}

\begin{document}
	\title{A short proof of the existence of a minor-universal countable planar graph}
	\author{George Kontogeorgiou}
	\maketitle
	
	An element of a class of graphs is \emph{minor-universal} for the class if every other element of the class is its minor. In \hspace{-1mm}\cite{DK}, Diestel and K\"uhn constructed a minor-universal countable planar graph. Their construction is rather involved. The purpose of this note is to present a significantly simpler construction and proof, based on a suggestion of Georgakopoulos (\hspace{-1mm}\cite{G}, p.7) to modify the construction of \cite{Gagain}. The idea is as follows: 
	
	\begin{enumerate}[i.]
	\item $\mathcal{G}_0$ is a copy of $C_3$ embedded in $\mathbb{S}^2$;
	\item for $n>0$, $\mathcal{G}_n$ is the plane graph obtained from $\mathcal{G}_{n-1}$ in the following way: in each face $F$ of $\mathcal{G}_{n-1}$ we place a vertex $v_F$, we connect $v_F$ to each vertex $u\in\partial F$ with an edge, and we subdivide each edge $\{v_F,u\}$ exactly $n$ times; 
	\item we define $\mathcal{G}:=\bigcup_{n=0}^{\infty}\mathcal{G}_n$. 
	\end{enumerate}

	\begin{figure}[!h]
		\centering
		\includegraphics[width=80mm]{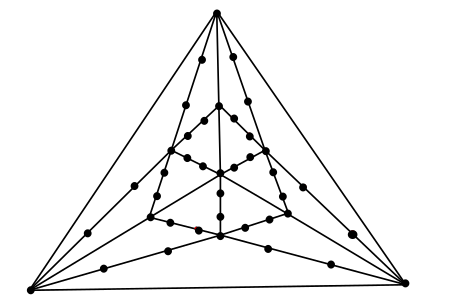}
		\caption{$\mathcal{G}_2$ is formed by gluing two copies of this graph at the outer triangle.} 
	\end{figure}

	In this note I prove:
	
	\begin{theorem}\label{main}
		The graph $\mathcal{G}$ is minor-universal for countable planar graphs.
	\end{theorem}

	By its construction, $\mathcal{G}$ is countable and planar. It remains to show that it has every other countable planar graph as a minor. This follows directly from the following two lemmata:
	
	\begin{lemma}\label{start}
		Every countable planar graph is a minor of a countable, sub-cubic, 2-connected planar graph.
	\end{lemma}

	\begin{proof}
		Let $G\in \mathcal{C}$. It is simple to construct a countable planar graph $G'\succ G$ in $\mathcal{C}$ that is 2-connected and has finite maximum degree by blowing up each vertex to a locally finite tree and taking the fattening, for details see (\hspace{-0.5mm}\cite{G}, Chapter 4). We blow up each vertex of $G'$ to a finite trivalent tree to get the promised graph $G''$.
	\end{proof} 
	
	\begin{lemma}\label{finish}
		Every countable, sub-cubic, 2-connected planar graph is a minor of $\mathcal{G}$.
	\end{lemma}

	The rest of this note is devoted to proving Lemma \autoref{finish}. 
	
	\section*{Proof of Lemma \autoref{finish}} 
	
	An \emph{ear decomposition} of a countable graph $G$ is a sequence $(G_n)_{n\in \mathbb{N}}$ of subgraphs of $G$ such that:
	\begin{enumerate}[i.]
		\item $G_0$ is a cycle;
		\item for $n>0$, either $G_n=G_{n-1}$ or $G_n$ is a simple graph obtained from $G_{n-1}$ by connecting two distinct vertices of $G_{n-1}$ with a new path (called an \emph{ear} in this context);
		\item $G=\bigcup_{n=0}^{\infty}G_n$.
	\end{enumerate}

	We first prove two claims that are used in our proof.
	 
	 \begin{claim} \label{decomp}
	 	Every 2-connected countable graph $G$ with at least three vertices admits an ear decomposition.
	 \end{claim}
 
 	\begin{proof}
 		We order the edges of $G$. By the infinite version of Menger's Theorem (\hspace{-0.5mm}\cite{H}), there exists a cycle $C$ that contains $e_0$. Set $G_0:=C$. Given $G_{n-1}$, we determine $G_n$ as follows. If $G_{n-1}=G$, then set $G_n:=G$. Otherwise, let $e$ be the edge outside of $G_{n-1}$ with the smallest index. It is a well-known exercise that there exists a cycle $C_e$ in $G$ that contains both $e$ and $e_0$. Let $P_e$ be the subpath of $C_e$ that meets $G_{n-1}$ only at its endpoints and contains $e$. Set $G_n:=G_{n-1}\cup P_e$. To see that $G=\bigcup_{n=0}^{\infty}G_n$, note that $e_n\in G_n$.   
 	
 	\end{proof}
 
 	\begin{remark} \label{facial}
 		If $G$ is a plane graph, then the $n^{th}$ ear of its ear decomposition lies in a face of $G_{n-1}$.
 	\end{remark}
 
 	An $n$-\textit{diameter} of $\mathcal{G}$ is a maximal path contained in $\mathcal{G}_{n+1}\setminus\mathcal{G}_n$. An $n$-\textit{slice} of $\mathcal{G}$ is a subgraph that is bounded by (but does not contain the edges of) the boundary of a face of $\mathcal{G}_n$ and perhaps one of the diameters of that face, and does not contain $\mathcal{G}_0$. An $n$-\textit{piece} is a subgraph of $\mathcal{G}$ of the form $\overline{F}\cap\mathcal{G}$ or $(\overline{F}\cap\mathcal{G})\setminus\mathcal{G}_n$ for some face $F$ of $\mathcal{G}_n$. Recall that, given a separator $S$ of a graph $G$ and a connected component $K\subseteq G\setminus S$, the \emph{torso} of $K$ is the induced subgraph of $G$ on the vertices $V(K)\cup S$.

 	\begin{figure}[!h]
 		\centering
 		\includegraphics[width=60mm]{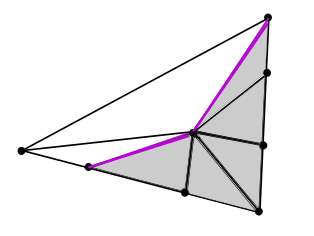}
 		\caption{A purple $2$-diameter bounding a grey area that contains a $2$-slice.} 
 	\end{figure}

	\begin{claim} \label{slice}
		Every $n$-slice is $(n+3)$-connected.
	\end{claim}
 
 	\begin{proof}
 	Let $J$ be an $n$-slice and $S$ be a separator of $J$. We say that a connected component is $N$-\emph{coarse} for $N\ge n$ if its torso in $J$ is the union of a finite set of $N$-pieces. Since $S$ is finite, it is entirely contained in some $\mathcal{G}_N$ for $N\ge n$, so it splits $J$ into $N$-coarse components. Therefore $S$ must have order at least $N+3\ge n+3$, since $N+3$ is half the length of an $N$-diameter rounded up, which is the shortest boundary that an $N$-coarse component can have with its complement in $J$. 
 	\end{proof}
 
 Recall that an \textit{inflated copy} (briefly denoted as $IG$) of a graph $G=(V,E)$ is a graph $H$ together with a \textit{map} $\phi:G\rightarrow H$ such that for every $v\in V$ and $e\in E$ with ends $\{x,y\}$, the following are true:
 
 \begin{enumerate}[i.]
 	\item $\phi(v)$ is a connected subgraph of $H$;
 	\item $\phi(e)$ is an open path between $\phi(x)$ and $\phi(y)$;
 	\item the images of the elements of $G$ are pairwise disjoint;
 	\item $\phi(G)=H$.
 \end{enumerate}

A graph $G$ is a minor of a graph $G'$ if and only if $G'$ contains an $IG$ as a subgraph. We are now ready to prove Lemma $\autoref{finish}$:
 	
 	\begin{proof}
 		Let $G$ be a plane graph as in the statement of Lemma \autoref{finish}. We will construct a subgraph $H$ of $\mathcal{G}$ that is an $IG$. In particular, we will define $H$ as the union of a sequence $H_n$ of $IG_n$, where $(G_n)_{n\in\mathbb{N}}$ is an ear decomposition of $G$ (Claim \autoref{decomp}), such that for every $n\in \mathbb{N}$, $v\in V(G)$, and $e\in E(G)$, we have for the corresponding maps that $\phi_n(v)\subseteq \phi_{n+1}(v)$ and $\phi_n(e)=\phi_{n+1}(e)$. This will allow us to define $\phi:G\rightarrow H$ by setting $\phi:=\bigcup_{n=0}^{\infty}\phi_n$; the reader is invited to check that this yields an IG. To satisfy these constraints, we also require a consistent way to correspond abstract faces of $G_n$ to slices of $\mathcal{G}$. To keep notation simple, we overload our maps $\phi_n$ to also perform this function.
 		
 		The following definition is key in our analysis: given a face $F$ of $G_n$, we denote by $V_F$ the set of vertices of $\partial F$ that are ends of ears contained in $F$. Note that, since $G$ is sub-cubic, each vertex can be the end of at most one ear.      
 		
 		We construct inductively the sequence $\phi_n$ with the additional useful requirements that for each $v\in V(G_n)$, $\phi_n(v)$ is a path, and that for each face $F$ of $G_n$, the slice $\phi_n(F)$ is $|V_F|$-connected and intersects $H_n$ only at one end vertex of each of the paths $\{\phi_n(v)|v\in\partial F\}$ and in the correct cyclic order.  
 		
 		For the base case, suppose that $G_0$ has length $l_0\ge 3$. We consider two disjoint face cycles $C_0$ and $C'_0$ of the graph $\mathcal{G}_{l_0}$, contained in the same ($l_0-2$)-piece, noting that each has length $2l_0+3>l_0$. By Menger's Theorem and Claim \autoref{slice}, we may also consider $l_0$ disjoint paths  inside the piece joining $C_0$ and $C'_0$. We define $\phi_0$ by mapping the vertices of $G_0$ to these paths in the correct cyclic order, each edge of $G_0$ to an appropriate (open) subpath of $C_0$, and the two faces $F_0$ and $F'_0$ of $G_0$ to the two $l_0$-slices $J_0$ and $J'_0$ bounded by $C_0$ and $C'_0$, respectively. Both $J_0$ and $J'_0$ are $(l_0+3)$-connected, and $l_0+3>l_0=|V_{F_0}|+|V_{F'_0}|$, so $F_0$ is mapped to a $|V_{F_0}|$-connected slice and $F'_0$ is mapped to a $|V_{F'_0}|$-connected slice.     
 		
 		For the inductive step, suppose that we have defined $\phi_{n-1}$ so that each vertex $v$ is mapped to a path $\phi_{n-1}(v)$ and each face $F$ is mapped to a $|V_F|$-connected slice $\phi_{n-1}(F)$ that intersects $H_{n-1}$ only at one end vertex of each of the paths $\{\phi_{n-1}(v)|v\in\partial F\}$ and in the correct cyclic order. 
 		
 		Let $F$, of boundary length say $l$, be the unique (Remark \autoref{facial}) face of $G_{n-1}$ such that the $n^{th}$ ear $G_n\setminus G_{n-1}$, of length say $L$, lies within $F$, splitting it to two faces $F'$ and $F''$. Let $f$ be a face of $\mathcal{G}_N$ for some $N>n$ such that $f$ is bounded by the slice $\phi_{n-1}(F)$, has diameter length at least $L$, and all its slices have connectivity at least $2L+l$ (e.g. take $N=2L+l-3$) (Claim \autoref{slice}).
 		
 		 \begin{figure}[!h]
 			\centering
 			\includegraphics[width=120mm]{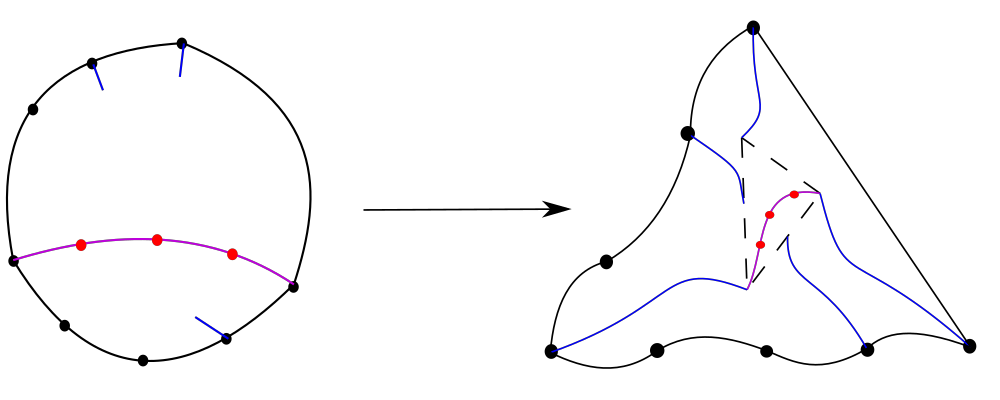}
 			\caption{A step in the construction of $\phi$.} 
 		\end{figure} 
 		
 		By the inductive hypothesis, $\phi_{n-1}(F)$ is $|V_F|$-connected. By Menger's Theorem, there exist disjoint paths $(P_v)_{v\in V_F}$ in $\phi_{n-1}(F)$ such that $P_v$ begins at the vertex $\phi_{n-1}(v)\cap \phi_{n-1}(F)$ and ends at some $p(v)\in \partial f$. In particular, since $\mathcal{G}$ is a plane graph, $p$ preserves the cyclic order of $\phi_{n-1}(V_F)$. We also denote by $D$ the diameter of $f$ with ends $p(x)$ and $p(y)$, where $x$ and $y$ are the ends of the ear $G_n\setminus G_{n-1}$. 
 		
 		For $v\in V_F$ we set $\phi_n(v)=\phi_{n-1}(v)\cup P_v$. The vertices of the ear $G_n\setminus G_{n-1}$ we map arbitrarily (but in the correct order) to single vertices of $D$. Each edge of the ear is mapped to the subpath of $D$ between the images of its ends. The faces $F'$ and $F$ are mapped to the corresponding slices induced on $f$ by $D$. Every other vertex, edge, or face of $G_n$ has the same image under $\phi_n$ as under $\phi_{n-1}$.
 		
 		Firstly, it follows from the construction that $\phi_n:G_n\rightarrow H_n$ is an $IG_n$. 
 		
 		Moreover, for $v\in V_F$, by the inductive hypothesis, $\phi_{n-1}(v)$ is a path. Also, $P_v$ is a path and, again by the inductive hypothesis, $\phi_{n-1}(v)\cap P_v$ is a single vertex. Therefore, $\phi_n(v)$ is a path. Trivially, the same holds for vertices in $V(G_n)\setminus V_F$. 
 		
 		Additionally, both $F'$ and $F''$ are mapped to slices of connectivity $2L+l$. Since $|V_{F'}|+|V_{F''}|\leq |\partial F'|+|\partial F''|=2L+l$, both of these slices are adequately connected. We have also ensured that they intersect the appropriate vertex images under $\phi_n$ only at one end vertex each and in the correct cyclic order. These same properties hold, by the inductive hypothesis, for every other face image of $G_n$. This concludes the induction, and the proof of Lemma \autoref{finish}.
 	  
 		\end{proof} 
 	
 	\begin{remark}
 	The above construction is quite manageable, and perhaps can be modified to tackle the problem of minor-universal elements for 2-complexes embeddable in $\mathbb{S}^3$, mentioned in (\hspace{-1mm}\cite{G}, p.12).
 	\end{remark}
 	
 	\bibliography{minor-universal}
 	\bibliographystyle{plain}
	 
\end{document}